\documentclass[11pt]{amsart}
\pdfoutput=1

\usepackage{amssymb}
\usepackage{mathtools}
\usepackage{microtype}
\usepackage{aliascnt}

\usepackage[numbers]{natbib}
\usepackage{enumitem}

\usepackage{hyperref}
\hypersetup{
  pdftitle={Maximum principles for the relativistic heat equation},
  pdfauthor={Evan Miller and Ari Stern},
  bookmarksopen=false
}

\theoremstyle{plain} %--default
\newtheorem{theorem}             {Theorem}  [section]
\newaliascnt{lemma}{theorem}
\newtheorem{lemma}      [lemma]{Lemma}
\aliascntresetthe{lemma}
\newaliascnt{corollary}{theorem}
\newtheorem{corollary}      [corollary]{Corollary}
\aliascntresetthe{corollary}
\newaliascnt{conjecture}{theorem}
\newtheorem{conjecture}      [conjecture]{Conjecture}
\aliascntresetthe{conjecture}

\theoremstyle{definition}
\newtheorem{definition} [theorem]{Definition}

\theoremstyle{remark}
\newtheorem{remark} [theorem]{Remark}

\begin{document}
\title[Maximum principles for the relativistic heat equation]{Maximum
  principles\\ for the relativistic heat equation}
\author{Evan Miller}
\email{ep.miller@mail.utoronto.ca}
\address{University of Toronto}

\author{Ari Stern}
\email{astern@math.wustl.edu}
\address{Washington University in St.~Louis}

\begin{abstract}
  The classical heat equation is incompatible with relativity, since
  the strong maximum principle allows for disturbances to propagate
  instantaneously. Some authors have proposed limiting the propagation
  speed by adding a linear hyperbolic correction term, but then even a
  weak maximum principle fails to hold. We study a more recently
  introduced relativistic heat equation, which replaces the Laplace
  operator by a quasilinear elliptic operator, and show that strong
  and weak maximum principles hold for stationary and time-varying
  solutions, respectively, as well as for sub- and
  supersolutions. Moreover, by transforming the equation into an
  equivalent form, related to the mean curvature operator, we prove
  even stronger tangency and comparison principles.
\end{abstract}

\maketitle

\section{Introduction}

\subsection{Background}

It is well known that the classical heat equation,
\begin{equation}
  \label{eqn:heatEquation}
  u _t = \Delta u ,
\end{equation}
allows for disturbances to propagate with infinite speed, in violation
of special relativity.  This is a direct consequence of the
\emph{strong maximum principle}, which states that if $u$ attains an
interior maximum at some positive time, then it must be constant for
all previous times. This implies that, if a constant solution is
perturbed locally, then the perturbation is instantly detected at
arbitrarily distant points.

To model heat conduction relativistically, various authors have
suggested replacing the parabolic heat equation by a linear hyperbolic
equation, such as the \emph{telegraph equation},
\begin{equation}
  \label{eqn:telegraphEquation}
  c ^{ - 2 } u _{ t t } + u _t = \Delta u ,
\end{equation} 
where the constant $c > 0 $ denotes the speed of light. Indeed, this
equation obeys the relativistic ``speed limit'' and reduces to
\eqref{eqn:heatEquation} in the limit as $ c \rightarrow \infty
$.
(For example, see~\citet{GuPi1968}.)  However, as a hyperbolic
equation, it differs from \eqref{eqn:heatEquation} in at least two
important respects. First, it lacks the regularity and smoothing
properties of \eqref{eqn:heatEquation}: solutions of
\eqref{eqn:telegraphEquation} need not be smooth and can even contain
``thermal shock waves'' (cf.~\citet{Tzou1989}).  Second, it does not
even satisfy a \emph{weak maximum principle}: for example,
\eqref{eqn:telegraphEquation} allows for two ``heat waves'' traveling
towards one another to combine into a larger wave, in violation of
this principle.

\citet{Brenier2003} introduced an alternative approach to relativistic
heat conduction, based on optimal transportation theory. This results
in the quasilinear parabolic equation
\begin{equation}
  \label{eqn:rhec2}
  u _t = \operatorname{div}  \frac{ u {D}  u }{ \sqrt{ u ^2 + c ^{ - 2 } \lvert
      {D}  u \rvert ^2 } }.
\end{equation} 
which we call the \emph{relativistic heat equation}. If $ u > 0 $, we
again obtain \eqref{eqn:heatEquation} in the limit as
$ c \rightarrow \infty $. (By contrast, if $ u < 0 $, the limit is
instead the ill-posed backward heat equation, $ u _t = - \Delta u $.)
Brenier's starting point was to observe that \eqref{eqn:heatEquation}
can be understood as gradient descent of the Boltzmann entropy
functional $ \int u \log u \,\mathrm{d}x $, where the gradient is with
respect to the Wasserstein metric, corresponding to optimal
transportation with the nonrelativistic kinetic energy cost function
$ k (v) = \frac{1}{2} \lvert v \rvert ^2 $
(cf.~\citet{JoKiOt1998}). Using instead the relativistic cost function
\begin{equation*}
  k (v) =
  \begin{cases}
    \biggl( 1 - \sqrt{ 1 - \frac{ \lvert v \rvert ^2 }{ c ^2 } }
    \biggr) c ^2 ,& \lvert v \rvert \leq c ,\\
    \infty ,& \lvert v \rvert > c ,
  \end{cases}
\end{equation*} 
yields \eqref{eqn:rhec2} as the gradient descent equation for
Boltzmann entropy. For the remainder of this paper, we set $ c = 1 $.

It is natural to ask whether the relativistic heat equation
\eqref{eqn:rhec2} satisfies a weak maximum principle, similar to that
satisfied by \eqref{eqn:heatEquation} but not by
\eqref{eqn:telegraphEquation}. The purpose of the present paper is to
answer this question in the affirmative, and to give some related
results on maximum principles for the relativistic heat equation.

\subsection{Outline of the paper}

First, in \autoref{sec:elliptic}, we consider stationary solutions to
the relativistic heat equation, which we call \emph{relativistically
  harmonic} functions (by analogy with harmonic functions, which are
stationary solutions to the classical heat equation); we also consider
subsolutions and supersolutions, which we call relativistically
\emph{subharmonic} and \emph{superharmonic}, respectively. A crucial
component of our analysis is \autoref{lem:w}, which transforms the
quasilinear elliptic operator in \eqref{eqn:rhec2} into a more
convenient form, related to the mean curvature operator. Using this
transformed formulation, we prove a strong maximum (minimum) principle
for subsolutions (supersolutions), as well as even stronger tangency
and comparison principles. Finally, we prove that relativistically
harmonic functions are real analytic, and use this to give another,
more elementary proof of the strong maximum principle.

Next, in \autoref{sec:parabolic}, we consider time-dependent solutions
of the relativistic heat equation, along with subsolutions and
supersolutions. While finite propagation speed (i.e., relativity)
precludes the possibility of a strong maximum or minimum principle,
much less an even stronger tangency principle, we show that comparison
and weak maximum/minumum principles do hold. Finally, we discuss one
possible direction for future work, in which a stronger
maximum/minimum principle and tangency principle might be shown to
hold on light cones, which would still be consistent with finite
propagation speed.

\section{The elliptic case}
\label{sec:elliptic}

\subsection{Solutions, subsolutions, and supersolutions}

\begin{definition}
  Given an open set $ U \subset \mathbb{R}^n $, a function
  $ u \in C ^2 (U) $ with $ u > 0 $ is \emph{relativistically
    subharmonic} if
  \begin{equation}
    \label{eqn:subharmonic}
    \operatorname{div} \frac{ u \,{D} u }{ \sqrt{ u ^2 + \lvert {D}  u
        \rvert ^2 } } \geq  0 ,
  \end{equation} 
  \emph{relativistically superharmonic} if 
  \begin{equation}
    \label{eqn:superharmonic}
    \operatorname{div} \frac{ u \,{D} u }{ \sqrt{ u ^2 + \lvert {D}  u
        \rvert ^2 } } \leq   0 ,
  \end{equation} 
  and \emph{relativistically harmonic} if
  \begin{equation}
    \label{eqn:harmonic}
    \operatorname{div} \frac{ u \,{D} u }{ \sqrt{ u ^2 + \lvert {D}  u
        \rvert ^2 } } = 0 ,
  \end{equation} 
  i.e., if it satisfies both \eqref{eqn:subharmonic} and
  \eqref{eqn:superharmonic}.
\end{definition}

For the purposes of the subsequent analysis, these expressions are
less than ideal. In addition to the $ u > 0 $ restriction, the
coefficients depend on both $u$ and $ {D} u $, while several results
on quasilinear elliptic operators require dependence on $ {D} u $
alone.  However, by making the substitution $ w = \log u $, we now
show that it is possible to obtain an equivalent formulation that is
valid for all real-valued $w$, and where the coefficients depend only
on $ {D} w $, not on $w$ itself.

\begin{lemma}
  \label{lem:w}
  Given $ u \in C ^2 (U) $ with $ u > 0 $, let $ w = \log u $, and
  define the quasilinear operator
  \begin{equation*}
    Q w = \Delta w - \frac{ {D} ^2 w ( {D} w , {D} w ) }{ 1 + \lvert
      {D} w \rvert ^2 } + \lvert {D} w \rvert ^2 .
  \end{equation*} 
  Then $u$ is relativistically subharmonic if and only if
  \begin{equation}
    \label{eqn:subharmonicw}
    Q w \geq 0 ,
  \end{equation} 
  superharmonic if and only if 
  \begin{equation}
    \label{eqn:superharmonicw}
    Q w \leq 0 ,
  \end{equation} 
  and harmonic if and only if 
  \begin{equation}
    \label{eqn:harmonicw}
    Q w = 0 .
  \end{equation} 
\end{lemma}

\begin{proof}
  Since $ u > 0 $, observe that
  \begin{equation*}
    \frac{ u \,{D} u }{ \sqrt{ u ^2 + \lvert {D} u \rvert ^2 } } =
    \frac{ u \frac{  {D} u }{ u } }{ \sqrt{ 1 + \frac{ \lvert {D} u
          \rvert ^2 }{  u ^2} } } = \frac{ e ^w {D} w }{ \sqrt{ 1 +
        \lvert {D} w \rvert ^2 } } .
  \end{equation*} 
  Taking the divergence of this expression, a short calculation gives
  \begin{equation*}
    \operatorname{div} \frac{ e ^w {D} w }{ \sqrt{ 1 + \lvert {D} w
        \rvert ^2 } } = \frac{ e ^w \Bigl[ \bigl( \Delta w + \lvert
      {D} w \rvert ^2 \bigr) \bigl( 1 + \lvert {D} w \rvert ^2 \bigr)
      - {D} ^2 w ( {D} w , {D} w ) \Bigr] }{ \bigl( 1 + \lvert {D} w
      \rvert ^2 \bigr) ^{ 3/2 } } .
  \end{equation*} 
  Substituting this into each of
  \eqref{eqn:subharmonic}--\eqref{eqn:harmonic}, dividing by
  $ e ^w \bigl( 1 + \lvert {D} w \rvert ^2 \bigr) ^{ - 1/2 } > 0 $,
  and rearranging yields
  \eqref{eqn:subharmonicw}--\eqref{eqn:harmonicw}, respectively.
\end{proof}

\begin{lemma}
  \label{lem:elliptic}
  The quasilinear operator $Q$ is non-uniformly elliptic.
\end{lemma}

\begin{remark}
  This essentially amounts to the non-uniform ellipticity of the
  well-studied mean curvature operator
  $ \mathfrak{M} w = \bigl( 1 + \lvert {D} w \rvert ^2 \bigr) \Delta w
  - {D} ^2 w ( {D} w , {D} w ) $
  (cf.~\citet{GiTr2001}), but the demonstration is sufficiently brief
  that we give it here.
\end{remark}

\begin{proof}
  The principal part of $Q$ can be written as
  $ a ^{ i j } ( {D} w ) {D} _i {D} _j w $, where
  \begin{equation*}
    a ^{ i j } ( p ) = \delta ^{ i j } - \frac{ p ^i p ^j }{ 1 +
      \lvert p \rvert ^2 } ,
  \end{equation*} 
  with $ \delta ^{ i j } $ denoting the Kronecker delta.  The matrix
  $ p ^i p ^j $ is symmetric with rank $1$, so its only nonzero
  eigenvalue (having multiplicity $1$) is its trace,
  $ \lvert p \rvert ^2 $. Diagonalizing, it follows that
  $ a ^{ i j } ( p ) $ has eigenvalues
  \begin{equation*}
    \lambda _1 = \frac{ 1 }{ 1 + \lvert p \rvert ^2 } , \qquad \lambda
    _2 = \cdots = \lambda _n = 1 ,
  \end{equation*} 
  which are positive; however, $ \lambda _1 $ is not bounded uniformly
  away from zero.
\end{proof}

\subsection{Maximum/minimum, tangency, and comparison principles} We
begin by giving the strong maximum and minimum principles for
subsolutions and supersolutions, respectively.

\begin{theorem}
  \label{thm:ellipticmax}
  Suppose $ U \subset \mathbb{R}^n $ is open and connected. 
\begin{enumerate}[label=(\roman*)]
\item If $ w \in C ^2 (U) $ satisfies $ Q w \geq 0 $ and attains an
  interior maximum in $U$, then $w$ is constant.

\item If $ w \in C ^2 (U) $ satisfies $ Q w \leq 0 $ and attains an
  interior minimum in $U$, then $w$ is constant.
\end{enumerate} 
Consequently, if $ u \in C ^2 (U) $ is relativistically
subharmonic (superharmonic), then it satisfies the corresponding
strong maximum (minimum) principle.
\end{theorem}

\begin{proof}
  Using \autoref{lem:elliptic} and the fact that the elliptic operator
  contains no zeroth-order terms in $w$, the statements for $w$ follow
  immediately by applying Hopf's strong maximum principle. (On the
  applicability of Hopf's principle to nonlinear elliptic inequalities
  by ``freezing'' coefficients, which is lesser-known than the linear
  case, see~\citet{PuSe2007}.) The corresponding statement for $u$
  then follows by the monotonicity of $ w \mapsto e ^w = u $.
\end{proof}

In fact, using the fact that $Q$ is independent of $w$, we can
establish an even stronger ``tangency'' principle, which implies the
preceding strong maximum/minimum principle as a special case.

\begin{theorem}
  \label{thm:elliptictangency}
  Suppose $ w , w ^\prime \in C ^2 (U) $ satisfy
  $ Q w \geq Q w ^\prime $ and $ w \leq w ^\prime $ in $U$.  If
  $ w = w ^\prime $ at some point $ x \in U $, then
  $ w \equiv w ^\prime $ in $U$.  Consequently, if
  $ u , u ^\prime \in C ^2 (U) $ satisfy
  \begin{equation}
    \label{eqn:uinequality}
    \operatorname{div} \frac{ u \,{D} u }{ \sqrt{ u ^2 + \lvert {D}  u
        \rvert ^2 } } \geq \operatorname{div} \frac{ u ^\prime \,{D} u
      ^\prime }{ \sqrt{ ( u ^\prime ) ^2 + \lvert {D}  u ^\prime 
        \rvert ^2 } } 
  \end{equation} 
  and $ u \leq u ^\prime $ in $U$, and if $ u = u ^\prime $ at some
  point $ x \in U $, then $ u \equiv u ^\prime $ in $U$.
\end{theorem}

\begin{proof}
  Since $Q$ is elliptic, and since its coefficients are independent of
  $w$ and continuously differentiable (in fact, analytic) in
  $ {D} w $, the result is obtained by applying the tangency principle
  for nonlinear elliptic operators in \citet[Theorem
  2.1.3]{PuSe2007}. Finally, as in the proof of
  \autoref{thm:ellipticmax}, the corresponding statement for $u$ and
  $ u ^\prime $ follows by the monotonicity of the exponential
  function.
\end{proof}

\begin{remark}
  Note that \autoref{thm:ellipticmax} is a special case of
  \autoref{thm:elliptictangency}, where we take either
  $ w ^\prime \equiv M $, the maximum attained by $w$, or
  $ w \equiv m ^\prime $, the minimum attained by $ w ^\prime $.
\end{remark}

Finally, we establish a comparison principle, which relates the values
of $w$ and $ w ^\prime $ on the boundary with those in the
interior. Taking either $w$ or $ w ^\prime $ to be constant, we get
the weak maximum principle as a special case.

\begin{theorem}
  Suppose $ w, w ^\prime \in C ^2 (U) \cap C ( \overline{ U } ) $
  satisfy $ Q w \geq Q w ^\prime $ in $U$. If $ w \leq w ^\prime $ on
  $ \partial U $, then $ w \leq w ^\prime $ in $U$. Consequently, if
  $ u , u ^\prime \in C ^2 (U) \cap C ( \overline{ U } ) $ satisfy the
  inequality \eqref{eqn:uinequality} in $U$ and $ u \leq u ^\prime $
  on $ \partial U $, then $ u \leq u ^\prime $ in $U$.
\end{theorem}

\begin{proof}
  Again, since $Q$ is elliptic, and since its coefficients are
  independent of $w$ and continuously differentiable (in fact,
  analytic) in $ {D} w $, we may apply the nonlinear elliptic
  comparison principle in \citet[Theorem 2.1.4]{PuSe2007} or its
  quasilinear counterpart in \citet[Theorem 10.1]{GiTr2001}. As above,
  the corresponding statement for $u$ and $ u ^\prime $ follows by the
  monotonicity of the exponential function.
\end{proof}

\subsection{Analyticity and an elementary proof of the strong maximum
  principle for relativistically harmonic functions}

Finally, we show that relativistically harmonic functions are
analytic, and we use this to give an elementary, self-contained proof
of the strong maximum principle for solutions (but not sub- or
supersolutions) using analyticity rather than the machinery of the
Hopf principle.

\begin{theorem}
  \label{thm:analytic}
  Every solution $ w \in C ^2 (U) $ of $ Q w = 0 $ is real
  analytic. Consequently, every relativistically harmonic function $u$
  is real analytic.
\end{theorem}

\begin{proof}
  Analyticity of $w$ follows from \autoref{lem:elliptic} by classical
  elliptic theory (e.g.,~\citet{Hopf1932,Morrey1958a}), since the
  coefficients depend analytically on $ {D} w $.  Analyticity of $u$
  then follows immediately from \autoref{lem:w}.
\end{proof}

We now give an alternative proof that, if a relativistically harmonic
function attains an interior maximum, then it must be constant. (The
proof of the minimum principle is essentially identical, modulo the
direction of the corresponding inequalities.)

\begin{proof}
  Suppose $u$ attains an interior maximum $ u (x) = M $ at some
  $ x \in U $.  Since $u$ is real analytic by \autoref{thm:analytic},
  the maximum at $x$ must be isolated unless $u$ is constant. In the
  latter case, we're done. Otherwise, the maximum is isolated, so
  there exists an open ball $ B ( x, r ) \subset U $ of radius
  $ r > 0 $ such that $ {D} u (y) \cdot ( y - x ) \leq 0 $ for all
  $ y \in B ( x, r ) $.  That is, $ {D} u \cdot \nu \leq 0 $ on
  $ \partial B ( x, s ) $ for all $ 0 < s < r $, where
  $ \nu (y) = \frac{ 1 }{ s } (y - x) $ is the outer unit normal at
  $ y \in \partial B ( x, s ) $. Now, since
  \begin{equation*}
    \operatorname{div}  \frac{ u \,{D}  u }{ \sqrt{ u ^2 + \lvert {D}  u
        \rvert ^2 } } = 0 ,
  \end{equation*} 
  integrating over $ B ( x, s ) $ and applying the divergence theorem
  implies
  \begin{equation*}
    \int _{ \partial B ( x, s ) } \frac{ u }{ \sqrt{ u ^2 + \lvert
        {D}  u \rvert ^2 } } {D}  u \cdot \nu \,\mathrm{d} S = 0 .
  \end{equation*} 
  However, since
  $ \frac{ u }{ \sqrt{ u ^2 + \lvert {D} u \rvert ^2 } } > 0 $ and
  $ {D} u \cdot \nu \leq 0 $ on $ \partial B ( x, s ) $, it follows
  that $ {D}  u \cdot \nu = 0 $ on $ \partial B ( x, s ) $. Since
  this holds for all $ 0 < s < r $, the function $u$ is constant along
  radii of $ B ( x, r ) $, and hence constant on $ B ( x, r )
  $. Explicitly, if $ y \in B ( x, r ) $ and $ \nu = \nu (y) $, then
  \begin{equation*}
    u (y) - u (x) = \int _0 ^{ \lvert y - x \rvert } 
    \frac{\mathrm{d}}{\mathrm{d}s} u ( x + \nu s )
    \,\mathrm{d}s = \int _0 ^{ \lvert y - x \rvert } {D}  u ( x +
    \nu s ) \cdot \nu \,\mathrm{d}s = 0 ,
  \end{equation*} 
  so $ u \equiv M $ on $ B ( x, r ) $. Hence, the nonempty and
  relatively closed set $ u ^{-1} \bigl( \{ M \} \bigr) \cap U $ is
  also open, so the result follows by the assumption that $U$ is
  connected.
\end{proof}

\section{The parabolic case}
\label{sec:parabolic}

\subsection{Solutions, subsolutions, and supersolutions}

We now turn our attention to time-dependent solutions of the
relativistic heat equation, along with subsolutions and
supersolutions. Throughout this section, we denote the spacetime
domain by $ ( x, t ) \in U _T \times ( 0, T ] $, where
$ U \subset \mathbb{R}^n $ is an open set and $ T > 0 $, and the
parabolic boundary of $ U _T $ by
$ \Gamma _T = \overline{ U _T } - U _T $. We say that
$ u \in C ^2 _1 ( U _T ) $ if $u(x,t)$ is $ C ^2 $ in $x$ and $ C ^1 $
in $t$ for all $ (x,t) \in U _T $. (This is consistent with the
notation found, e.g., in \citet{Evans2010}.)

\begin{definition}
  Given $ u \in C ^2 _1 (U _T ) $ with $ u > 0 $, we say that $u$ is a
  \emph{subsolution} of the relativistic heat equation if
  \begin{equation*}
    u _t - \operatorname{div} \frac{ u \,{D} u }{ \sqrt{ u ^2 + \lvert
        {D} u \rvert ^2 } } \leq 0,
  \end{equation*} 
  a \emph{supersolution} if
  \begin{equation*}
    u _t - \operatorname{div} \frac{ u \,{D} u }{ \sqrt{ u ^2 + \lvert
        {D} u \rvert ^2 } } \geq  0,
  \end{equation*} 
  and a \emph{solution} if
  \begin{equation*}
    u _t - \operatorname{div} \frac{ u \,{D} u }{ \sqrt{ u ^2 + \lvert
        {D} u \rvert ^2 } } = 0,
  \end{equation*} 
  i.e., if $u$ is both a subsolution and supersolution.
\end{definition}

As in the previous section, we use the change of variables
$ w = \log u $ to transform the problems above into a more convenient
form, where the elliptic coefficients depend only on $ {D} w $ rather
than on $w$ itself.

\begin{lemma}
  Given $ u \in C ^2 _1 ( U _T ) $ with $ u > 0 $, let $ w = \log u $,
  and define the quasilinear operator
  \begin{equation*}
    \widetilde{ Q } w = \frac{ Q w }{ \sqrt{ 1 + \lvert {D} w \rvert
        ^2 } } .
  \end{equation*}
  Then $u$ is a subsolution of the relativistic heat equation if and
  only if
  \begin{equation*}
    w _t - \widetilde{ Q } w \leq 0 ,
  \end{equation*} 
  a supersolution if and only if
  \begin{equation*}
    w _t - \widetilde{ Q } w \geq 0 ,
  \end{equation*} 
  and a solution if and only if
  \begin{equation*}
    w _t - \widetilde{ Q } w = 0 .
  \end{equation*} 
\end{lemma}

\begin{proof}
  Since $ w = \log u $, we have
  \begin{equation*}
    u _t = u \frac{ u _t }{ u } = e ^w w _t ,
  \end{equation*} 
  so the result follows by the calculation given in the proof of
  \autoref{lem:w}.
\end{proof}

\subsection{Comparison and weak maximum/minimum principles}

The relativistic heat equation, as its name suggests, does not permit
instantaneous propagation of disturbances (cf.~\citet{AnCaMaMo2006}),
and hence we cannot expect a strong maximum principle or tangency
principle to hold in the time-dependent case. However, we now show
that a comparison principle does still hold, which implies a weak
maximum/minimum principle as an immediate corollary.

\begin{theorem}
  \label{thm:paraboliccomparison}
  Suppose $ w , w ^\prime \in C _1 ^2 ( U _T ) \cap C ( \overline{ U
    _T } ) $ satisfy
  \begin{equation*}
    w _t - \widetilde{ Q } w  \leq w
    ^\prime _t - \widetilde{ Q } w ^\prime 
  \end{equation*} 
  in $ U _T $. If $ w \leq w ^\prime $ on $ \Gamma _T $, then
  $ w \leq w ^\prime $ in $ U _T $. Consequently, if $ u , u ^\prime
  \in C _1 ^2 ( U _T ) \cap C ( \overline{ U _T } ) $ satisfy
  \begin{equation*}
    u _t - \operatorname{div} \frac{ u \,{D} u }{ \sqrt{ u ^2 + \lvert
        {D} u \rvert ^2 } } \leq  u ^\prime _t - \operatorname{div}
    \frac{ u ^\prime \,{D} u ^\prime }{ \sqrt{ (u ^\prime ) 2 + \lvert
        {D} u ^\prime \rvert ^2 } }
  \end{equation*} 
  in $ U _T $, and if $ u \leq u ^\prime $ on $ \Gamma _T $, then
  $ u \leq u ^\prime $ in $ U _T $.
\end{theorem}

\begin{proof}
  The proof is essentially a parabolic adaptation of the quasilinear
  elliptic comparison principle (\citet[Theorem 10.1]{GiTr2001}).

  First, we rearrange the inequality to obtain
  \begin{equation*}
    ( w - w ^\prime ) _t - ( \widetilde{ Q } w - \widetilde{ Q } w
    ^\prime ) \leq 0 .
  \end{equation*}
  Next,
  \begin{multline*}
    \widetilde{ Q } w - \widetilde{ Q } w ^\prime = \widetilde{a} ^{ i
      j } ( {D} w ) {D} _i {D} _j ( w - w ^\prime ) \\
    + \bigl[ \widetilde{ a } ^{ i j } ( {D} w ) - \widetilde{ a } ^{ i
      j } ( {D} w ^\prime ) \bigr] {D} _i {D} _j w ^\prime +
    \widetilde{ b } ( {D} w ) - \widetilde{ b } ( {D} w ^\prime ) .
  \end{multline*} 
  Letting $ z = w - w ^\prime $, we then define the linear 
  operator
  $ L z = a ^{ i j } (x,t) {D} _i {D} _j z + b ^i (x,t) {D} _i z $ so that
  \begin{align*}
    a ^{ i j } (x,t) &= \widetilde{a} ^{ i
                     j } ( {D} w ),\\
    b ^i (x,t) {D} _i z &= \bigl[ \widetilde{ a } ^{ i j } ( {D} w ) -
                        \widetilde{ a } ^{ i
                        j } ( {D} w ^\prime ) \bigr] {D} _i {D} _j w ^\prime +
                        \widetilde{ b } ( {D} w ) - \widetilde{ b } (
                        {D} w ^\prime ) . 
  \end{align*} 
  (In the last step, it is crucial that $ \widetilde{ a } ^{ i j } $
  and $ \widetilde{ b } ^i $ are independent of $w$, $ w ^\prime $.)
  Hence, $ \widetilde{ Q } w - \widetilde{ Q } w ^\prime = L z $, so
  we have
  \begin{equation*}
    z _t - L z \leq 0 
  \end{equation*} 
  in $ U _T $, and $ z \leq 0 $ on $ \Gamma _T $. By the parabolic
  weak maximum principle (cf.~\citet[Section 7.1, Theorem
  8]{Evans2010}), we conclude that $ z \leq 0 $ in $ U _T $, i.e.,
  $ w \leq w ^\prime $, which completes the proof.

  Finally, as in the previous section, the corresponding statement for
  $u$ and $ u ^\prime $ follows by the monotonicity of the exponential
  function.
\end{proof}

\begin{corollary}
  Let $ w \in C _1 ^2 ( U _T ) \cap C ( \overline{ U _T } ) $.
  \begin{enumerate}[label=(\roman*)]
  \item If $ w _t - \widetilde{ Q } w \leq 0 $, then
    $ \max _{ \overline{ U } _T } w = \max _{ \Gamma _T } w $.
  \item If $ w _t - \widetilde{ Q } w \geq 0 $, then
    $ \min _{ \overline{ U } _T } w = \min _{ \Gamma _T } w $.
  \end{enumerate} 
  The corresponding statements hold for subsolutions and
  supersolutions of the relativistic heat equation.
\end{corollary}

We also mention another corollary of
\autoref{thm:paraboliccomparison}, which establishes monotonicity and
uniqueness properties for solutions to the relativistic heat equation

\begin{corollary}
  Suppose
  $ w , w ^\prime \in C _1 ^2 ( U _T ) \cap C ( \overline{ U _T } ) $
  are both solutions to the relativistic heat equation. If
  $ w \leq w ^\prime $ on $ \Gamma _T $, then $ w \leq w ^\prime $ in
  $ U _T $. Furthermore, if $ w \equiv w ^\prime $ on $ \Gamma _T $,
  then $ w \equiv w ^\prime $ in $ U _T $.
\end{corollary}

\begin{proof}
  Since $w$ and $ w ^\prime $ are solutions,
  $ w _t - \widetilde{ Q } w = w ^\prime _t - \widetilde{ Q } w
  ^\prime = 0 $.
  Hence, the first statement follows immediately by applying
  \autoref{thm:paraboliccomparison}, while the second statement
  follows by applying it again with $w$ and $ w ^\prime $
  interchanged.
\end{proof}

\subsection{Conjectured strong maximum/minimum and tangency principles
  in light cones}

We conclude with a brief discussion of possible directions in which
the results of this section might be strengthened in the future---but
which would also require the introduction of new techniques.

The strong maximum (minimum) principle for the classical heat equation
says that, if a subsolution (supersolution) $u$ attains an interior
maximum (minimum) at a point $ ( x, t ) \in U _T $, then $u$ must be
constant on $ U _t = U \times (0, t] $. As mentioned above, we cannot
hope for this to hold for the relativistic heat equation---at least
not for the cylinder $ U _t $. However, it seems likely that an
analogous statement might hold on the backwards light cone
\begin{equation*}
  U _{ ( x, t ) } = \bigl\{ ( \xi, \tau ) \in U _T : \lvert x - \xi
  \rvert \leq t - \tau \bigr\} ,
\end{equation*} 
which contains only points that can affect $ ( x, t ) $ without
violating relativity. (For $ c \rightarrow \infty $, this approaches
the cylinder $ U _t $.) The following conjecture states a version of
this principle restricted to this backwards light cone.

\begin{conjecture}
  If $ w \in C ^2 _1 ( U _T ) \cap C ( \overline{ U _T } ) $ satisfies
  $ w _t - \widetilde{ Q } w \leq 0 $ and attains an interior maximum
  $M$ at a point $ ( x, t ) \in U _T $, then $w \equiv M $ in the
  backwards light cone $ U _{ ( x, t ) } $. Likewise, if $w$
  satisfies $ w _t - \widetilde{ Q } w \geq 0 $ and attains an
  interior minimum $m$ at a point $ ( x, t ) \in U _T $, then
  $w \equiv m $ in $ U _{ ( x, t ) } $.
\end{conjecture}

We also expect that an even stronger tangency principle might hold on
the backwards light cone, as follows.

\begin{conjecture}
  Suppose
  $ w , w ^\prime \in C ^2 _1 ( U _T ) \cap C ( \overline{ U _T } ) $
  satisfy
  \begin{equation*}
    w _t - \widetilde{ Q } w \leq w ^\prime _t - \widetilde{ Q } w
    ^\prime 
  \end{equation*}
  and $ w \leq w ^\prime $ in the backwards light cone
  $ U _{ (x,t) } $ of a point $ ( x, t ) \in U _T $. If
  $ w ( x, t ) = w ^\prime ( x, t ) $, then $ w \equiv w ^\prime $ in
  $ U _{ (x,t) } $.
\end{conjecture}

\section*{Acknowledgments}

This work was supported in part through the ARTU program at Washington
University in St.~Louis, funded by NSF CAREER Award DMS-1055897, for
which we wish to express our thanks to \'Alvaro Pelayo. Additional
support was provided by a grant from the Simons Foundation (\#279968
to Ari Stern).

\

\end{document}